\documentclass{amsart}
\usepackage{hyperref}
\usepackage{graphicx}
\usepackage{amscd}
\usepackage{verbatim}

\newtheorem{theorem}{Theorem}[section]

\theoremstyle{definition}
\newtheorem{definition}[theorem]{Definition}
\newtheorem{proposition}[theorem]{Proposition}
\newtheorem{example}[theorem]{Example}

\newtheorem{corollary}[theorem]{Corollary}

\theoremstyle{remark}
\newtheorem{remark}[theorem]{Remark}

\numberwithin{equation}{section}

\begin{document}

\title[Arithmetic of LG models]{On the arithmetic of Landau-Ginzburg model of a
certain class of threefolds}

\author{Genival Da Silva Jr.}

\address{\newline Department of Mathematics \newline Imperial College London\newline London, SW7 2AZ, UK}
\email{g.junior@imperial.ac.uk}


\begin{abstract}
We prove that the Ap\'ery constants for a certain class of Fano threefolds can be obtained as a special value of a higher normal function.
\end{abstract}

\maketitle

\section{Introduction}
The application of normal functions in areas peripheral to Hodge theory has emerged as a topic of research over the last decade \cite{BVK15},\cite{BV},\cite{DK},\cite{Hain},\cite{Ker},\cite{Mel}; areas related to physics have accounted for much of this growth. The goal of this paper is to use normal functions to give a `motivic' meaning to constants arising in quantum differential equations associated to a certain class of Landau-Ginzburg models.\par
In \cite{BVK15}, there is a explicit computation of a higher normal function associated with the Landau-Ginzburg mirror of a rank $4$ Fano threefold, which turns out to be the value of a Feynman Integral. We want to present a similar approach, but instead of a Feynman integral, we will express some Ap\'ery constants (\cite{ASZ08},\cite{GOL09},\cite{Gal},\cite{GGI}) in terms of special values of the associated higher normal functions.\par 
Landau-Ginzburg models are the natural object for `mirrors' of Fano manifolds; more precisely, mirror symmetry relates a Fano variety with a dual object, which is a variety equipped with a non-constant complex valued function. For example, a LG model for $\mathbb{P}^2$ is a family of elliptic curves and more generally, the LG model of a Fano $n$-fold is a family of Calabi-Yau $(n-1)$-folds. In general, mirror symmetry relates symplectic properties of a Fano variety with algebraic ones of the mirror and vice versa.\par
In the following sections we will be mainly concerned with the Landau-Ginzburg models for a special class of threefolds, namely the ones whose associated local system is of rank three, with a single nontrivial involution exchanging two maximally unipotent monodromy points.  Looking at the classification in \cite{CCGGK}, one finds the short list $V_{12},V_{16},V_{18}$ and ``$R_1$'', where the first three are rank $1$ Fanos appearing in \cite{GOL09} and the latter is a rank $4$ threefold with $-K^3=24$ ($K$ the canonical divisor). The involutions for these LG models have essentially been described in \cite{GOL09} and \cite{BVK15}. In the presence of an involution, it is possible to move the degeneracy locus of a higher cycle from the fiber over $0$ to its involute, a property which we use for the construction of the desired normal function.\par  
Let $\mathbb{P}_{\Delta^\circ}$ be a toric degeneration of any of the varieties considered above; then each one of these will have a mirror Landau-Ginzburg model, which is a family of $K3$ surfaces in $\mathbb{P}_{\Delta}$, that can be constructed as follows. Let $\phi$ be a Minkowski polynomial for $\Delta$, then the family of $K3$ is: 
\begin{equation}
X_t := \overline{\{1-t\phi(\mathbf{x})=0\}}\subset\mathbb{P}_{\Delta}
\end{equation}
Let
\begin{equation}
\omega_t=\tfrac{1}{(2\pi i)^2}Res_{X_t}\left(\frac{\frac{dx_1}{x_1}\wedge\frac{dx_2}{x_2}\wedge\frac{dx_3}{x_3}}{1-t\phi}\right)
\end{equation} 
and $\gamma_t$ the invariant vanishing cycle about $t=0$. We define the period of $\phi$ by
\begin{equation}
\Pi_{\phi}(t)=\int_{\gamma_t}\omega_t=\sum a_nt^n
\end{equation}
where $a_n$ is the constant term of $\phi^n$. We say that $a_n$ is the period sequence of $\phi$.\\
Consider a polynomial differential operator $L=\sum F_k(t)P_k(D_t)$ where $P_k(D_t)$ is a polynomial in $D_t=t\frac{d}{dt}$, then $L\cdot  \Pi_{\phi}(t)=0$ is equivalent to a linear recursion relation. In practice, to compute $L$ one uses knowledge of the first few terms of the period sequence and linear algebra to guess the recursion relation. The operator $L$ is called a $\textit{Picard Fuchs operator}$.\par
\begin{example}
The Picard-Fuchs operator for the threefold $V_{12}$ is:
\begin{equation}
L = D^3-t(1+2D)(17D^2+17D+5)+t^2(D+1)^3
\end{equation}
\end{example}
More generally, one also gets the same linear recursion on the power-series coefficients $b_k$ of solutions of inhomogeneous equations $L(\,\cdot\,)=G$, $G$ a polynomial in $t$, for $n\geq \deg(G)$, where $n$ is the degree of $L$.
\begin{definition}[\cite{GOL09}]
Given a linear homogeneous recurrence $R$ and two solutions $a_n , b_n \in \mathbb{Q}$ of $R$ with $a_0=1,b_0=0,b_1=1$. If there is a Dirichlet character with associated $L$-function $L(x)$, and an integer $x_0>1$ such that:
\begin{equation}\label{apery_limit_def}
\lim \frac{b_n}{a_n}=cL(x_0) , \quad c\in \mathbb{Q}^*
\end{equation}
We say that \ref{apery_limit_def} is the Ap\'ery constant of $R$.
\end{definition}
When we have a family of Calabi-Yau manifolds, a common way to look for Ap\'ery constants is by considering the Picard-Fuchs equation. As described above, the coefficients of the power series expansion of the solutions of this equation satisfy a recurrence and in some cases the Ap\'ery constant exists, see \cite{ASZ08} for a wide class of examples. Beyond this ``classical'' case, we can also talk about quantum recurrences, which are recurrences arising from solutions of the Quantum differential equations satisfied by the quantum periods, which are defined using quantum products, see \cite{GOL07}.\par
In \cite{GOL09}, Golyshev uses quantum recurrences of the threefolds $V_{10},V_{12},V_{14},V_{16},V_{18}$ to find Ap\'ery constants; his method is basically to use a result of Beukers [\cite{GOL09}, Proposition 3.3] for the rational cases and apply a different approach for the non-rational ones. In the course of the proof of his results, he also describes the involution we mentioned above, but only for $V_{12},V_{16}$ and $V_{18}$. The main theorem of this manuscript is: 

\begin{theorem}\label{mainth}
Let $X$ be a Fano threefold, in the special class described above. Then there is a higher normal function $\mathcal{N}$, arising from a family of motivic cohomology classes on the fibers of the LG model, such that the Ap\'ery constant is equal to $\mathcal{N}(0)$.
\end{theorem}

As an immediate corollary of this result and Borel's theorem, the Ap\'ery constant for these cases must be a $\mathbb{Q}$-linear combination of $\zeta(3)$ and $(2\pi i)^3$, except for $V_{18}$, where we have a factor of $\sqrt{-3}$. This corollary provides a uniform conceptual explanation of the results in \cite{GOL09} and \cite{BVK15}.

\begin{remark} We note that throughout this paper, the cycle groups are taken modulo torsion ($\otimes \mathbb{Q}$).
\end{remark}

\subsection*{Acknowledgements} I thank my advisor Matt Kerr for sharing his ideas with me, C. Doran and A. Harder for discussions regarding this work, and the
two referees for helpful suggestions. The author acknowledges the travel support from NSF FRG Grant 1361147 and the support of CNPq program \textit{Science without borders.}

\section{Construction of the ``toric'' motivic classes}
We assume the reader is familiar with the basic notions of Toric geometry, see \cite{CK} for a brief review or \cite{CL} for a more comprehensive treatment. Let
\begin{equation}
\phi=\sum a_\textbf{m} \textbf{x}^\textbf{m} \in \mathbb{C}[x^{\pm 1},y^{\pm 1},z^{\pm 1}]
\end{equation}
be a Laurent polynomial with coefficients in $\mathbb{C}$ and $\Delta$ be the Newton polytope associated with $\phi$, which we will assume to be reflexive. (A list of all 3-dimensional reflexive polytopes is available at \cite{CCGGK}.) We briefly review the construction of the anti-canonical bundle and the facet divisors on the toric variety $\mathbb{P}_{\Delta}$. Let $x,y,z$ be the toric coordinates on $\mathbb{P}_\Delta$ and for each codimension $1$ face $\sigma \in \Delta(1)$, choose a point $\textbf{o}_\sigma \in \sigma$ with integral coordinates, and write $\mathbb{R}_{\sigma}^2$ for the $2$-plane through $\sigma$ . Then take a basis $\textbf{m}_1,\textbf{m}_{2}$ for the translate $(\mathbb{R}^2_\sigma\cap\mathbb{Z}^3) - \textbf{o}_\sigma$ and complete it to a basis $\textbf{m}_1,\textbf{m}_2,\textbf{m}_{3}$ for $\mathbb{Z}^3$ such that
\begin{equation}
\mathbb{R}_{\geq 0}\langle \pm \textbf{m}_1 , \pm \textbf{m}_{2} ,\textbf{m}_3 \rangle \supset \Delta - \textbf{o}_\sigma
\end{equation}
Change coordinates, by setting $x_j^{\sigma}=\textbf{x}^{\textbf{m}_j},j=1,2,3$. Consider the subset
\begin{equation}
\mathbb{D}_{\sigma}^{*}=\{x_1^{\sigma},x_2^{\sigma} \in \mathbb{C}^{*}\}\cap\{x_3^{\sigma}=0\}
\end{equation}
of $\mathbb{P}_{\Delta}$; let $\mathbb{D}_{\sigma}$ be the Zariski closure of $\mathbb{D}_{\sigma}^{*},$ and set
\begin{equation}
\mathbb{D}:=\sum_{\sigma \in \Delta(1)}[\mathbb{D}_{\sigma}] = \mathbb{P}_{\Delta}\backslash (\mathbb{C}^*)^3.
\end{equation}Henceforth we shall write $x,y,z$ for $x_1,x_2,x_3$.\par
A standard result in toric geometry is that the sheaf $\mathcal{O}(\mathbb{D})$ is ample and in case $\Delta$ is reflexive; it is also the anti-canonical sheaf for $\mathbb{P}_\Delta$, and hence $\mathbb{P}_\Delta$ is Fano in this case.\par
Given nowhere vanishing holomorphic functions $f_1,\ldots,f_n$ on a quasi-projective variety $Y$, we denote the higher Chow cycle given by the graph of the $f_j$ in $Y\times(\mathbb{P}^1)^n$ by $\langle f_1,\ldots,f_n\rangle \in CH^n(Y,n)$.
\begin{definition}
A $3$ dimensional Laurent polynomial $\phi$ is tempered if the symbol $\langle x^\sigma,y^\sigma \rangle_{D_{\sigma}^*} \in CH^2(D_{\sigma}^*,2)$ is trivial, for all facets $\sigma$, where $D_{\sigma}^*\subset\mathbb{D}_{\sigma}^*$ is the zero locus of the facet polynomial $\phi_\sigma=\textbf{x}^{-\textbf{o}_\sigma}\phi(\textbf{x})$.
\end{definition}
\begin{remark}\label{rmk01}
The definition above can be restated as follows: For $X_t$ a general $K3$ surface of the family induced by $\phi$, let $X_t^*=X_t\cap(\mathbb{C}^*)^3$; then $\phi$ is tempered if the image of the higher Chow cycle $\xi_t:=\langle x,y,z \rangle_{X_t^{*}}\in CH^3(X_t^*,3)$ under all residue maps vanishes.  (Equivalently, viewed as an element of Milnor $K$-theory $K^M_3(\mathbb{C}(X_t))$, $\xi_t$ belongs to the kernel of the Tame symbol, cf. \cite{Ker2}.)
\end{remark}
In this work, we will focus on a special class of Laurent polynomials, namely Minkowski polynomials. See \cite{ACGK} for the basic definitions and properties of Minkowski polynomials.
\begin{example}
Consider the Minkowski polynomial $\phi=x+y+z+(xyz)^{-1}$ with Newton polytope $\Delta$ with vertices $(1,0,0),(0,1,0),(0,0,1)$ and $(-1,-1,-1)$, see figure \ref{pic01}. Let $\sigma$ be the facet with vertices $(1,0,0),(0,1,0),(-1,-1,-1)$ and fix $(-1,-1,-1)$ as the 'origin' of the facet. Then clearly one possible choice of the new toric coordinates is:
\begin{equation}
\begin{split}
x^\sigma=x^2 y z\\
y^\sigma=x y^2 z\\
z^\sigma=x^{-1}
\end{split}
\end{equation}
Moreover $\mathbb{D}^*_\sigma=\{z^\sigma=0\}$, so that $D_{\sigma}^*$ is given by the zero locus of the facet polynomial $\phi_\sigma=1+x^\sigma+y^\sigma$. Therefore $Res_{D_{\sigma}^*}\langle x,y,z\rangle_{X_t^*} = Res_{z^{\sigma}=0}\langle x^{\sigma},y^{\sigma},z^{\sigma}\rangle_{X_t^*} = \langle x^{\sigma},y^{\sigma} \rangle_{D_{\sigma}^*} = \langle x^{\sigma},-1-x^{\sigma} \rangle=0$. Similarly, any other facet $\sigma$ of this polytope has the property that $\langle x^{\sigma},y^{\sigma} \rangle_{D_{\sigma}^*}=0$.
\end{example}
\begin{figure}\label{pic01}
\includegraphics{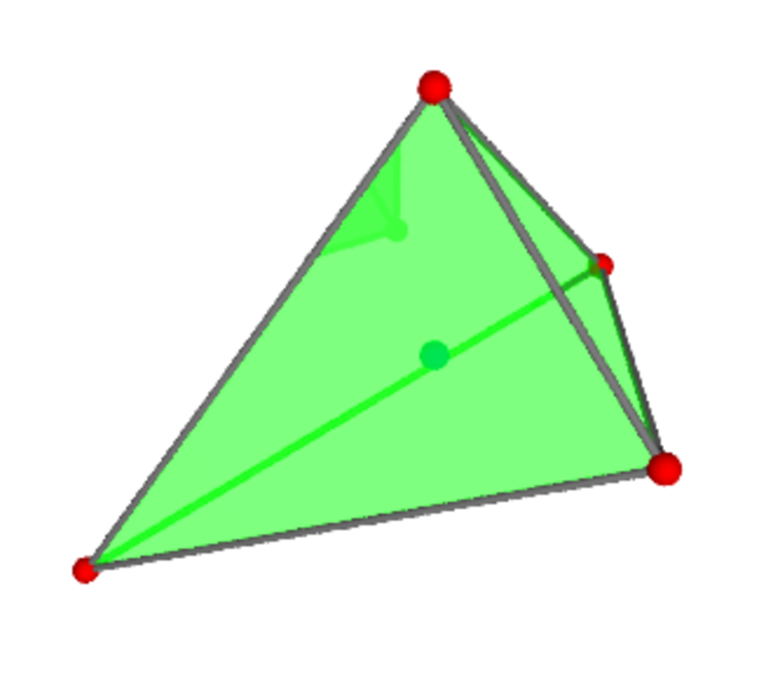}
\caption{Newton polytope for the Laurent polynomial $\phi=x+y+z+(xyz)^{-1}$. Taken from \cite{CCGGK}.}
\centering
\end{figure}
The fact that the symbol $\langle x^\sigma,y^\sigma \rangle_{D_{\sigma}^*}$ is trivial for all facets is not a coincidence; in fact, this is always the case for three-dimensional Minkowski polynomials. More precisely, we have:
\begin{proposition}
Every three-dimensional Minkowski polynomial is tempered.
\end{proposition}
\begin{proof}
In general, it is not true that every Laurent polynomial is tempered; one of the features of Minkowski polynomials is that they give rise to a decomposition in terms of rational irreducible subvarieties, a fact that will be strongly used below.  We use the equivalent definition of tempered as presented in remark \ref{rmk01}.\par

Noting that $D_{\sigma}:=\mathbb{D}_{\sigma}\cap X_t$ and $D=\mathbb{D}\cap X_t=\cup D_{\sigma}$ are independent of $t\neq 0$, and $X_t^*=X_t\setminus D$, let $\imath:D\to X_t$ and $\jmath:X_t^*\to X_t$ be the natural inclusions. The localization exact sequence for higher Chow groups reads:\begin{equation} \cdots \to CH^2(D,3)\overset{\imath_*}{\to} CH^3(X_t,3)\overset{\jmath^*}{\to}CH^3(X_t^*,3)\overset{Res_D}{\to}CH^2(D,2)\cdots\end{equation}Now in general, $D_{\sigma}$ is reducible, with components determined by the Minkowski decomposition of $\sigma$.  Write $D=\cup D_i$ as the resulting union of irreducible curves, and $D_i^*=D_i\setminus\cup_j (D_i\cap D_j)$.  By the localization sequence (for $D_i$), we have\begin{equation}\label{loc}CH^2(D_i,2)=\ker \left\{ CH^2(D_i^*,2)\overset{Res_{ij}}{\to} \oplus_j CH^1(D_i\cap D_j,1)\right\}.\end{equation}Since the edge polynomials of a Minkowski polynomial are cyclotomic,\footnote{in fact the roots are $\pm 1$} for every $i,j$ the composition\begin{equation}CH^3(X_t^*,3)\overset{Res_{i}}{\to}CH^2(D_i^*,2)\overset{Res_{ij}}{\to}\oplus_j CH^1(D_i\cap D_j,1)\end{equation}sends $\xi_t$ to zero.  By \eqref{loc}, we therefore have $Res_i\xi\in CH^2(D_i,2)$ for every $i$. Since in dimension  $3$ the irreducible pieces of a lattice Minkowski decomposition are either segments or triangles with no interior points, all the $D_i$ are rational and smooth.  Moreover, since both the Minkowski polynomial and the decomposition of the facet polynomials are defined over $\bar{\mathbb{Q}}$, the $D_i$ are rational over $\bar{\mathbb{Q}}$.  Now the $Res_i\xi$ are clearly defined over $\bar{\mathbb{Q}}$ (as the $Res_{\sigma}\xi_t=\langle x^{\sigma},y^{\sigma}\rangle$ are), and so belong to $CH^2(\mathbb{P}^1,2)=0$. The latter follows from using the localization sequence for the pair $(\mathbb{P}^1, Spec(\bar{\mathbb{Q}}))$.

Therefore $Res_i \xi_t$ is trivial, and $\phi$ is tempered by Remark \ref{rmk01}.\end{proof}

\begin{remark}
The notion of Minkowski polynomial for dimension greater than $3$ is not yet well understood. However, if we assume the lattice polytopes in the Minkowski decompositions of facets have no interior points, then the proof above will extend to dimension $4$, since we would still have rationality of the $D_i$ (as above), and no significant problems appear in the local-global spectral sequence for higher Chow groups.
\end{remark}

\section{The Higher normal function}
Recall that if $S$ is a smooth projective variety, then
\begin{equation}
H_{\mathcal{M}}^n(S,\mathbb{Q}(n))\cong CH^n(S,n)\cong Gr_\gamma^nK_n(S) .
\end{equation}\par
\noindent Not every member of our family $X_t$ is smooth, but we can still have an element in the motivic cohomology.  Such elements can be explicitly represented via higher Chow (double) complexes, so that we can still use standard formulas for Abel-Jacobi maps \cite[\S 8]{KL}:
\begin{equation}
AJ^{n,n}:H_{\mathcal{M}}^n(X_t,\mathbb{Q}(n))\rightarrow H^{n-1}(X_t,\mathbb{C}/\mathbb{Q}(n)) .
\end{equation}\par
The Landau-Ginzburg models for the threefolds $V_{12},V_{16},V_{18}$, and $R_1$, may be defined by (the Zariski closure of) the families $\{1-t\phi=0\}$, with $\phi$ given by:
\begin{equation}
\begin{split}
& V_{12} : \phi=\frac{(1+x+z)(1+x+y+z)(1+z)(y+z)}{xyz}\\
& V_{16} : \phi=\frac{(1+x+y+z)(1+z)(1+y)(1+x)}{xyz}\\
& V_{18} : \phi=\frac{(x+y+z)(x+y+z+xy+xz+yz+xyz)}{xyz}\\
& R_1: \phi=\frac{(1+x+y+z)(xyz+xy+xz+yz)}{xyz}
\end{split}
\end{equation}\par
\begin{figure}\label{pic02}
\includegraphics{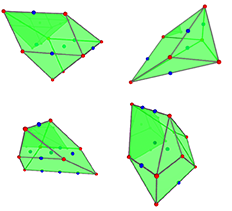}
\caption{Newton polytopes for  (top) $V_{18}, R_1$ and (bottom) $V_{12}, V_{16}$ respectively. Taken from \cite{CCGGK}.}
\centering
\end{figure}

\noindent As these families of K3s all have Picard rank 19, their Picard-Fuchs operators take the form $D_{PF} = \sum_{i=0}^3 F_k(t) (D_t)^k$, with $F_i(t)$ relatively prime polynomials.  We call $F_3(t)=:\sigma(D_{PF})$, which is taken to be monic, the {\it symbol} of $D_{PF}$.  In the four cases the symbols are \begin{equation} \label{symbols} t^2 - 34t + 1,\, t^2 - 24t + 16,\, t^2 - 18t - 27,\, \text{and}\, t^2 - 20t + 64,\end{equation} respectively.

We shall adopt the notation $\mathcal{X}\overset{\pi}{\to}\mathbb{P}^1$ for the total space of each family, obtained after a maximal projective triangulation of $\Delta^\circ$[\cite{DK},section 2.5], and $\mathcal{X}^\circ=\mathcal{X}\setminus X_0\overset{\pi^\circ}{\to}\mathbb{A}^1_{\frac{1}{t}}$ and $\mathcal{X}_\circ=\mathcal{X}\setminus X_\infty\overset{\pi_\circ}{\to}\mathbb{A}^1_t$, for restrictions. Henceforward, $X$ will denote any threefold in the list $V_{12}, V_{16}, V_{18}, R_1$.

\subsection*{Proof of theorem \ref{mainth}}
Associated to $X$ is a Newton polytope $\Delta$, and to the latter we associate a Minkowski polynomial $\phi$. The proposition above implies that $\phi$ is tempered, and being a Minkowski polynomial, it's also regular. By \cite[Remark 3.3(iii)]{DK}, the family of higher Chow cycles lifts to a class $[\Xi]\in H^3_{\mathcal{M}}(\mathcal{X}^{\circ},\mathbb{Q}(3))$, yielding by restriction a family of motivic cohomology classes $[\Xi_t]\in H^3_{\mathcal{M}}(X_t,\mathbb{Q}(3))$ on the Landau-Ginzburg model.  (On the smooth fibers these are just higher Chow cycles.)

The local system $\mathbb{V}=R^2_{tr}\pi_{*}\mathbb{Z}$(transcendental part of the second cohomology) associated to the Landau-Ginzburg model of $X$ has the following singular points:

\begin{itemize}
\item $V_{12}$ : $t= 0 , 17\pm 12\sqrt 2 , \infty$
\item $V_{16}$ : $t= 0 , 12\pm 8\sqrt 2 , \infty$
\item $V_{18}$ : $t= 0 , 9\pm 6\sqrt 3 , \infty$
\item $R_{1}$ : $t= 0 , 4, 16, \infty$
\end{itemize}

\noindent (Besides $0$ and $\infty$, these are just the roots of $\sigma(D_{PF})$.)

In each case, we have an involution $\iota(t)=\frac{M}{t}$, ($M=1, 16, -27, 64$), exchanging 2 singular points, say $t_1$ and $t_2$ with $0<|t_1|<|t_2|<\infty$. The involution $\iota$ gives then a correspondence $I\in Z^2(\mathcal{X}\times \iota^*\mathcal{X})$ which gives a rational isomorphism between $\mathbb{V}$ and $\iota^*\mathbb{V}$. Notice that the involution does not lift to the total space, as explained in (\cite{BVK15},section 3.3). Since $I$ induces an isomorphism, the vanishing cycle $\gamma_t$ at $t=0$ is sent to a rational multiple of the vanishing cycle $\mu_t$ at $t=\infty$. Hence in a neighborhood of $t=0$, we have:

\begin{equation}
\int_{\gamma_t}I^*\omega_{\iota(t)}=\int_{I_{*}\gamma_t}\omega_{\iota(t)}=n\int_{\mu_{\iota(t)}}\omega_{\iota(t)}, n\in \mathbb{Q}^*
\end{equation}

Moreover, as a section of the Hodge bundle\footnote{For more on the Deligne extension see \cite{GGK}}, $\omega_t$ has a simple zero at $t=\infty$ and no zero or poles anywhere else, since the degree of the Hodge bundle is 1 in this case\cite[Section V]{GGK3}. So $I^{*}\omega_{\iota(t)}=Ct\omega_t$, for some $C\in\mathbb{C}^{*}$. If we set $A(t)=\int_{\gamma_t}\omega_t$, then $A(0)=1$, and it follows from the residue theorem applied three times that

\begin{multline}\label{c_con}
C=\lim_{t\rightarrow 0} \frac{n}{(2\pi i)^2A(t)}\int_{\mu_{\iota(t)}} Res_{X_{\iota(t)}}\left(\frac{\frac{dx}{x}\wedge \frac{dy}{y}\wedge \frac{dz}{z}}{t-M\phi}\right) \\
 =-\frac{n}{M}Res^3_p \left( \frac{dx\wedge dy \wedge dz}{xyz\cdot \phi(x,y,z)} \right)\, ,
\end{multline}

\noindent where $p\in \text{sing}(X_{\infty})$ is the point to which $\mu_{\iota(t)}$ contracts to.  An explicit residue computation using SAGE\cite{SAGE} gives that $C$ is rational  in all cases except for $V_{18}$, where we have a rational multiple of $\sqrt{-3}$. Hence $\tilde{\omega}:=I^{*}\omega$ is a rational multiple of $t\omega$ in all cases except for $V_{18}$, where it is a $\sqrt{-3}$-multiple.

Now let $\tilde{\Xi}:=I^*\Xi\in H^3_{\mathcal{M}}(\mathcal{X}_{\circ},\mathbb{Q}(3))$ be the pullback of the cycle, with fiberwise slices $\tilde{\Xi}_t$.  If $AJ$ is the Abel-Jacobi map$\footnote{In smooth fibers, AJ takes a rather simple form in terms of currents, see \cite{KLM}}$ as above, then
\begin{equation}
AJ^{3,3}([\tilde{\Xi}_t])\in H^2(X_t,\mathbb{C}/\mathbb{Q}(3)).
\end{equation}

\noindent Taking $\mathcal{R}_t$ to be any lift of this class to $H^2(X_t,\mathbb{C})$, we may define a normal function by:
\begin{equation}
\mathcal{N}(t):= \langle \mathcal{R}_t , \omega_t \rangle
\end{equation}

\noindent By \cite[Prop. 4.1]{DK}, $\mathcal{N}(t)$ is well-defined in a open set containing the singular locus except the point $\infty$, thus it has a power series of radius of convergence $|t_2|>|t_1|$, where $0,t_1,t_2,\infty$ are the singular points of the local system.
\begin{proposition}\cite[Corollary 4.5]{DK}
Let $\mathcal{Y}(t)=\langle \tilde{\omega_t} , \nabla_{D_t}^2\omega_t \rangle$ be the Yukawa coupling and $\sigma(D_{PF})$ the symbol of the operator $D_{PF}$.Then
\begin{equation} \label{inhomog}
D_{PF}(\mathcal{N}(t) )= \sigma(D_{PF})\mathcal{Y}(t)
\end{equation}
\end{proposition}
\begin{proof}
We have that
\begin{equation}
D_t\langle \mathcal{R}_t , \omega_t\rangle = \langle \tilde{\omega_t} , \omega_t\rangle + \langle \mathcal{R}_t , \nabla_{D_t}\omega_t\rangle = \langle \mathcal{R}_t , \nabla_{D_t}\omega_t\rangle
\end{equation}
aplying $D_t$ again we have
\begin{equation}
D_t^{2}\langle \mathcal{R}_t , \omega_t\rangle = 
D_t\langle \mathcal{R}_t , \nabla_{D_t}\omega_t\rangle=
\langle \mathcal{R}_t , \nabla_{D_t}^2\omega_t\rangle
\end{equation}
since $\langle \tilde{\omega_t} ,\nabla_{D_t}\omega_t\rangle=0$ by Griffiths transversality. Finally, applying $D_t$ once more:
\begin{equation}
D_t^{3}\langle \mathcal{R}_t , \omega_t\rangle = 
D_t\langle \mathcal{R}_t , \nabla_{D_t}^2\omega_t\rangle = \langle \tilde{\omega_t} , \nabla_{D_t}^2\omega_t\rangle
\end{equation}
In our case, $D_{PF}$ is of the form $\sigma(D_{PF})D_t^{3} + \sum_{i=0}^2 F_k(t) (D_t)^k$, thus
\begin{equation}
D_{PF}(\mathcal{N}(t)) = \sigma(D_{PF})\mathcal{Y}(t) + \langle \mathcal{R}_t , D_{PF}\omega_t\rangle = \sigma(D_{PF})\mathcal{Y}(t)
\end{equation}
\end{proof}
Applying \cite[Rem. 4.4]{DK}, the right-hand side of \eqref{inhomog} takes the form $kt$, where (in view of \eqref{symbols}) $k=M\lim_{t\to 0} \frac{\mathcal{Y}(t)}{t}$. Denote by $\tilde{\mathbb{V}}=e^{-\frac{log(t)}{(2\pi i)}N}\mathbb{V}$ the canonical extension of $\mathbb{V}$, where $N$ is the log-monodromy around t=0. By \cite{LTY}, we have maximal unipotent monodromy at t=0. Let

\begin{equation}
N=
\begin{bmatrix}
    0 & 0 & 0 \\
    a & 0 & 0 \\
    c & b & 0 
\end{bmatrix}
,N^2=
\begin{bmatrix}
    0 & 0 & 0 \\
    0 & 0 & 0 \\
    ab & 0 & 0 
\end{bmatrix}
,a,b,c\in\mathbb{Q}^*
\end{equation}

\noindent and suppose $\mathbb{V}$ is generated by $\alpha , \beta, \gamma$. Then $\tilde{\mathbb{V}}$ is generated by
\begin{equation}\label{basis1}
\begin{split}
& \tilde{\alpha} = \alpha\\
& \tilde{\beta} = \beta -\frac{log(t)}{(2\pi i)}\alpha\\
& \tilde{\gamma} = \gamma -\frac{log(t)}{(2\pi i)}(c\alpha + b\beta) + \frac{1}{2}\frac{log^2(t)}{(2\pi i)^2}ab\alpha
\end{split}
\end{equation}

By writing $\omega_t$ in terms of \ref{basis1} and using it in the definition of $\mathcal{Y}(t)=\langle \tilde{\omega_t} , \nabla_{D_t}^2\omega_t \rangle$, we find that $k$ is $C$ times a rational constant, where $C$ is \ref{c_con}.  We conclude that
\begin{equation}\label{inhomog2}
D_{PF}(\mathcal{N}(t) )= kt
\end{equation}
\noindent where $k\in\mathbb{Q}^*$ in all cases, except for $V_{18}$, where it is a $\sqrt{-3}$-multiple.



Finally, if $A(t) = \sum a_n t^n$ is the period sequence, then $\tilde{B}(t) = \sum\tilde{b}_n t^n := \mathcal{N}(t) - 
A(t)\mathcal{N}(0)$ is another solution of the inhomogeneous Picard-Fuchs equation \eqref{inhomog2}, so that any multiples of $\{ a_n\}$ and $\{\tilde{b}_n\}$ satisfy the associated linear recurrence.  Since $\tilde{b}_1 = k$, we set $b_n:= - \sqrt{-3} \tilde{b}_n$ for $V_{18}$ and $b_n := - \tilde{b}_n$ otherwise, so that $B(t) = \sum b_n t^n$ has rational coefficients in all four cases.  We then have
\begin{itemize}
\item $\mathcal{N}(t)=\sum (a_n\mathcal{N}(0)-b_n)t^n $ for $V_{12},V_{16},R_{1}$
\item $\mathcal{N}(t)=\sum (a_n\mathcal{N}(0)-\frac{b_n}{\sqrt{-3}})t^n$ for $V_{18}$
\end{itemize}
Since the radii of convergence for the generating series of $a_n$ and $b_n$ are both $|t_1|<|t_2|$, while that of $a_n\mathcal{N}(0) - b_n$(or that of $a_n\mathcal{N}(0)-\frac{b_n}{\sqrt{-3}}$) is $|t_2|$, it follows that
\begin{itemize}
\item $\frac{b_n}{a_n}\rightarrow \mathcal{N}(0)$ for $V_{12},V_{16},R_{1}$
\item $\frac{b_n}{a_n}\rightarrow \sqrt{-3}\mathcal{N}(0)$ for $V_{18}$
\end{itemize}
\hfill$\square$
\begin{corollary}
For $V_{12},V_{16},R_{1}$, $\mathcal{N}(0)$ is (up to $\mathbb{Q}(3)$) a rational multiple of $\zeta(3)$. In the case $V_{18}$, the Ap\'ery constant is a rational multiple of $\frac{\pi^3}{\sqrt{3}}$.
\end{corollary}
\begin{proof}
The proof is a direct consequence of the following commutative diagram (See \cite[Example 8.21]{KL}):
\begin{equation}
\begin{CD}
H^3_{\mathcal{M}}(X_0,\mathbb{Q}(3)) @>\cong>> CH^3(Spec(\mathbb{Q}),5)\\
@VVAJ^{3,3} V @VVr_bV\\
H^2(X_0,\mathbb{C}/\mathbb{Q}(3)) @>>\cong> \frac{\mathbb{C}}{\mathbb{Q}(3)}
\end{CD}
\end{equation}
Where the lower isomorphism is the pairing with $\omega_0$ and $r_b$ is the Borel regulator. The Abel-Jacobi map then reduces to the Borel regulator and by Borel's theorem it has to be multiple of $\zeta(3)$. Note that since the Ap\'ery constant is real, for $V_{12},V_{16},R_{1}$, we have that $\mathcal{N}(0)$ is real and hence $\mathcal{N}(0)$ is a multiple of $\zeta(3)$. For $V_{18}$, we have that $\sqrt{-3}\mathcal{N}(0)$ is real and hence $\mathcal{N}(0)$ is imaginary, so it has to be a multiple of $(2i\pi)^3$ and therefore $\sqrt{-3}\mathcal{N}(0)$ is a multiple of $\frac{\pi^3}{\sqrt{3}}$.
\end{proof}

\begin{remark}
An explicit computation of $\mathcal{N}(0)$ for $R_1$ has been written in \cite{BVK15}; the computation for $V_{12}$ was done by M. Kerr and will be available in a forthcoming paper. Below we present the explicit computation of $\mathcal{N}(0)$ in the case $V_{16}$:
\end{remark}
\begin{example}
Consider $V_{16}$ which has a Minkowski polynomial given by $\phi=(x+1)(y+1)(z+1)(1+x+y+z)$; We change the coordinates $(x,y,z)\rightarrow (-x,-y,-z)$ to simplify the computations and use the same idea as \cite{BVK15}. The normal function $\mathcal{N}$ at $0$ takes the following form:
\begin{equation}
\mathcal{N}(0)=\int_\nabla R\{x,y,(1-x-y)\}
\end{equation}
Where $\nabla$ is the ``membrane'' $\nabla=$ \big\{(x,y) : $-1\leq y \leq 1$ , $-y\leq x \leq 1$ \big\}. We have:
\begin{equation}
\begin{split}
& \mathcal{N}(0) = \int_\nabla log(y)dlog\big(1-x-y\big)\wedge dlog(x)\\
& = \int_{-1}^1 log(y)\Big(\int_{-y}^1 \frac{dx}{x(1-x-y)}\Big)dy\\
& = \int_{-1}^1 log(y)\Big(\int_{-y}^1 \frac{dx}{x(1-y)} + \int_{-y}^1 \frac{dx}{(1-y)(1-x-y)}\Big)dy\\
& = 2\int_{-1}^1 log(y)\frac{log(-y)}{(1-y)}dy\\
& = 2[-(log(1-y)log(y)log(-y))_{-1}^1 + \int_{-1}^1 log(1-y)\frac{log(-y)+log(y)}{y}]\\
& = 2\int_{-1}^1 log(1-y)\frac{log(-y)+log(y)}{y}\\
& = 2[2 Li_{3}(y) -Li_{2}(y)log(-x^2)]_{-1}^1\\
& = 2[7\frac{\zeta(3)}{2} -i\tfrac{\pi^3}{4}]\\
& \equiv 7\zeta(3)\qquad \textrm{mod} \qquad \mathbb{Q}(3)\\
\end{split}
\end{equation}
where the $\mathbb{Q}(3)$ reflects the local ambiguity of $\mathcal{N}$ by a $\mathbb{Q}(3)$-period of $\tilde{\omega}$ (owing to the choice of lift $\mathcal{R}$).  Since the Ap\'ery constant is a real number, we normalize $\mathcal{N}$ locally by adding such a period to obtain $\mathcal{N}(0)=7\zeta(3)$.
\end{example}

\section{Concluding Remarks}
The proof of Theorem \ref{mainth} makes use of an involution of the family over $t\mapsto \pm\tfrac{M}{t}$ to produce a cycle with no residues on the $t=0$ fiber, but with nontorsion associated normal function.  That is, we use the involution to transport the residues of the cycle we \textit{do} know how to construct (via temperedness) to over $t=\infty$.

What is absolutely certain is that without a second maximally unipotent monodromy fiber (at $t=\infty$ in our four examples), such a normal function cannot exist.  This follows from injectivity of the topological invariant into \[Hom_{\text{MHS}}(\mathbb{Q}(0),H^3(\mathcal{X}^*,\mathbb{Q}(3)))\subset \oplus_{\lambda\in \Sigma} Hom_{\text{MHS}}(\mathbb{Q}(0),H_2(X_{\lambda},\mathbb{Q})),\] where $\Sigma\subset \mathbb{P}^1$ denotes the discriminant locus.  As an immediate consequence, nothing like Theorem \ref{mainth} can possibly hold for Golyshev's $V_{10}$ and $V_{14}$ examples.

While we could broaden the search to all local systems with more than one maximally unipotent monodromy point, those having an involution (or some other automorphism) represent our best chance for constructing cycles.  Though it is required to apply a couple of the tools of\cite{DK} as written, the $h^2_{tr}(X_t)=3$ assumption is perhaps less essential; if we drop this, there are many other LG local systems with ``potential involutivity''.  Inspecting data from \cite{CCGGK}, we see that the period sequences $35,49,52,53,55,59,60,62,97$ and $151$ have monodromies that suggest the presence of an involution. This is something we will investigate in future works.

Finally, we omitted one case with $h^2_{tr}(X_t)=3$ ad an involution, namely $B_4$ (cf. \cite{CCGGK}).  This is because there is a second involution, namely $t\mapsto -t$, wich probably rules out a meaninful Ap\'ery constant (as $|t_1|=|t_2|$).


\bibliographystyle{amsplain}

\end{document}